\documentclass[10pt,oneside,a4paper]{amsart}
\usepackage{latexsym}
\usepackage{amsfonts,amsmath,amssymb,indentfirst}
\newcommand{\beq}{\begin{equation}}
\newcommand{\eeq}{\end{equation}}
\newcommand{\bdism}{\begin{displaymath}}
\newcommand{\edism}{\end{displaymath}}

\newcommand{\setR}{\mathbb R}

\newtheorem{theorem}{Theorem}[section]

\author{\scshape Luca Fabrizio Di Cerbo}
\title{\bf A Ricci nilsoliton is nongradient}
\begin{document}

\thanks{*Supported in part by a Renaissance Technologies Fellowship.}
\address{Department of Mathematics, SUNY, Stony Brook, NY 11794-3651,
USA} \email{luca@math.sunysb.edu}
\maketitle

\pagenumbering{arabic}
In this brief note, we clarify that a Ricci nilsoliton cannot be of gradient type. The key is the rigidity theorem for gradient homogeneous Ricci solitons proved by Petersen and Wylie in \cite{Pet}. Recall that a gradient Ricci soliton is a Riemannian metric $g$ togheter with a function $f$ such that  
\begin{align}
{\rm Ric}_{g}+{\rm Hess}f=\lambda g,
\end{align}
where $\lambda$ is the soliton constant.

In \cite{Lau}, Lauret proves the existence of many left invariant Ricci solitons on nilpotent Lie groups. The first explicit construction of Lauret solitons has been obtained by Baird and Danielo in \cite{Danielo}. In particular they show that the soliton structure on  ${\rm Nil}^{3}$ is of nongradient type. Remarkably this is the first example of nongradient Ricci soliton. In \cite{Luca}, it is proved that any left invariant gradient Ricci soliton must have a nontrivial Euclidean de Rham factor. As an application of this result it is shown that any generalized Heisenberg Lie group is a nongradient left invariant Ricci soliton. In \cite{Pet}, Petersen and Wylie study the rigidity of gradient Ricci solitons with symmetry. Recall that a simply connected gradient Ricci soliton is called rigid if it is isometric to $N\times\setR^{k}$ where $N$ is an Einstein manifold and $f=\frac{\lambda}{2}\left|x\right|^{2}$ on the Euclidean factor. Among many other results, Petersen and Wylie prove the following optimal result for homogeneous gradient Ricci solitons.
\begin{theorem}(Petersen-Wylie)
All homogeneous gradient Ricci solitons are rigid.
\end{theorem}
Using this nice result we can now prove the following optimal generalization of theorem 3.7 in \cite{Luca}.
\begin{theorem}\label{rigidity}
A nilsoliton cannot be of gradient type.
\end{theorem}
\begin{proof}
Let $\left\lbrace\mathfrak{n}, \left\langle , \right\rangle\right\rbrace $ be a nilpotent soliton metric Lie algebra. Now, if the soliton is assumed to be of gradient type it must be isometric to a Riemannian product $N\times\setR^{k}$  where $N$ is an Einstein manifold. This implies that the Ricci tensor of our soliton is semi-definite. This is absurd with that fact any left invariant metric over a noncommutative nilpotent Lie group must have mixed Ricci curvatures, see theorem 2.4 in \cite{Milnor}. 
\end{proof}
\section{Final remarks}
Theorem \ref{rigidity} together with the results in \cite{Luca} provide a fairly complete description of the geometry of a Ricci nilsoliton. It would be now interesting to classify all the nilpotent Lie algebras that admit a soliton structure. In particular it would be interesting to understand if the characteristically nilpotent condition is the only obstruction to the existence of a soliton metric. For more details we refer to \cite{Lau}, \cite{Luca}.

\end{document}